\documentclass[12pt,final]{amsart}
\usepackage{amsmath,amssymb}

\usepackage{cite}

\usepackage{color}
\usepackage{soul,graphicx}

\usepackage[color]{showkeys}
\definecolor{refkey}{rgb}{0,0,1}
\definecolor{labelkey}{rgb}{1,0,0}

\newcommand{\ol}{\overline}

\newcommand{\tht}{\theta}

\newcommand{\eq} [1] {\begin{equation}\label{#1}\quad}
\newcommand{\en} {\end{equation}}

\newcommand{\diag}{\mathop{\rm diag}}
\newcommand{\rank}{\mathbb{\rm rank}}

\newcommand{\const}{\mathop{\rm Const}}

\newcommand{\card}{\mathop{\rm card}}

\newcommand{\tm}{\times}

\newcommand{\wdt}{\widetilde}
\newcommand{\tld}{\tilde}

\newcommand{\bT}{\mathbb{T}}
\newcommand{\bC}{\mathbb{C}}

\newcommand{\calU}{\mathcal{U}}
\newcommand{\calN}{\mathcal{N}}
\newcommand{\calR}{\mathcal{R}}

\newcommand{\mbo}{\mathbf{0}}

\newtheorem{theorem}{\bf  Theorem}
\newtheorem{lemma}{\bf  Lemma}

\newtheorem{corollary}{\bf \sc Corollary}

\begin{document}
\begin{center}
{\bf On non-optimal spectral factorizations}\\[4mm]

L. Ephremidze, I. Selesnick, and I. Spitkovsky
\end{center}

\vskip+0.3cm
\small{{\bf Abstract.} For a given Laurent polynomial matrix function $S$, which is positive definite on the unit circle in the complex plane, we consider all possible polynomial spectral factors of $S$ which are not necessarily invertible inside the unit circle.}

\section{Introduction}

Spectral factorization  is the process by which a positive matrix-valued
function $S$, on the unit circle in $\bC$, is expressed in the form
\begin{equation}\label{1}
S(e^{i\tht})=S_+(e^{i\tht})S_+^*(e^{i\tht}),
\end{equation}
 for a certain analytic matrix-valued function $S_+$.
Such factorization plays a crucial role in the solution of various applied problems
in system theory \cite{Kai99} and wavelet analysis \cite{Daub}. In the polynomial case, which is most widespread situation
in practise, the matrix spectral factorization theorem is formulated as follows:

\begin{theorem} Let
\begin{equation}\label{11}
 S(z)=\sum_{n=-N}^N C_nz^n
\end{equation}
be an $m\times m$ Laurent polynomial matrix function $(${}$C_n\in \mathbb{C}^{m
\times m}$ are matrix coefficients$)$ of degree $N$ $(${}$C_N\not=0${}$)$ which is positive definite
almost everywhere on $\mathbb{T}:=\{z\in \mathbb{C}: |z|=1\}$.
Then it admits a factorization
\begin{equation}\label{2}
 S(z)=S_+(z)\wdt{S_+}(z),\;\;\;\;\;\;z\in\mathbb{C}\backslash\{0\},
\end{equation}
where
\begin{equation}\label{3}
S_+(z)= \sum_{n=0}^N A_nz^n,\;\;\;\;A_n\in \mathbb{C}^{m \times m}
\end{equation}
is nonsingular inside $\mathbb{T}$,
$\det\,S_+(z)\not=0$ when $|z|<1$, and
$$
\wdt{S_+}(z)=\overline{S_+\left(1/\overline{z}\right)}^{\,{}_T}=\sum_{n=0}^N
 A^*_nz^{-n}
 $$ is its adjoint, $A_n^*={\overline{A}_n}^{\,{}_T}$,
 $n=0,1,\ldots,N$,
$($respectively, $\wdt{S_+}$ is analytic and nonsingular outside
$\mathbb{T}$ including infinity$)$.

$S_+$ is unique up to a constant right unitary multiplier $U$.
\end{theorem}

In the scalar case, $m=1$, the above result is known as the
Fej\'{e}r-Riesz lemma and can be easily proved by considering the
zeroes of $S(z)$.

The matrix spectral factorization \eqref{1} was first established by Wiener \cite{Wie57} in a general case
for any integrable matrix-valued function $S$ with an integrable logarithm of the determinant, $\log \det S\in
L_1({\mathbb T})$. In this case, the spectral factor $S_+$ is an outer analytic matrix function (see, e.g., \cite{EL10}) with entries from the Hardy space $H_2$.  Wiener proved this theorem by using a linear prediction theory of multi-dimensional stochastic processes. A little bit later, by using the same methods,
Rosenblatt \cite{Ros} showed  that $S_+$ is a polynomial whenever $S$ is a Laurent polynomial. Youla \cite{Youla} provided the first constructive proof of the theorem. Since then, many different simplified proofs of Theorem 1 have appeared in the literature (see \cite{DGK}, \cite{Callier}, \cite{CG91}, \cite{EJL09}). A very simple proof, which relies on elementary complex analysis and linear algebra, appeared in \cite{E}.

A requirement on the spectral factor \eqref{3} to be nonsingular, i.e. invertible, inside $\bT$ is essential for most applications. Such factorization is called {\em optimal} (or sometimes of {\em minimum-phase} especially in the scalar case, $m=1$), because the modulus of the determinant achieves maximal possible value in that case (see, e.g., \cite[Th. 17.17]{Ru}). The uniqueness of an optimal spectral factor $S_+$ is valid in a more strong sense then is explicitly formulated in Theorem 1. Namely, if $R_+(z)= \sum_{n=0}^\infty B_nz^n$, $B_n\in \mathbb{C}^{m \times m}$, is any analytic matrix function with entries from the Hardy space $H_p$, for some $p>0$, such that $R(e^{i\tht})=R_+(e^{i\tht})R_+^*(e^{i\tht})$ for a. a. $\tht\in[0,2\pi)$, then there exist a constant unitary matrix $U$, such that $S_+(z)=R_+(z)\cdot U$ (see, e.g., \cite{EJL09} for a simple proof), i.e. we do not have to require a priori  that $R_+$ is a polynomial of the same degree $N$ as $S_+$, but get it as a consequence.

If we multiply $S_+$ in \eqref{1} by a bounded analytic matrix function $\calU$ which is unitary a.e. on $\bT$, then the product
\begin{equation}\label{21}
R_+=S_+\calU
\end{equation}
still will be, in general, a non-optimal spectral factor of $S$, however, it might not be a polynomial, or even a rational matrix function. On the other hand, every non-optimal spectral factor of \eqref{11} with entries from the Hardy space $H_p$, for some $p>0$, can be written in the form \eqref{21} (it can be proved exactly in the same manner as the above uniqueness statement, see \cite{EJL09}). Youla was the first who mentioned that non-optimal spectral factorizations can also be applied in some network synthesis problems (see \cite[p. 182]{Youla}).

In the theory of wavelets, non-optimal factorizations also play a role in the construction of suitable bases. The construction of the well-known compactly-supported orthonormal Daubechies' wavelet bases for $L_2(R)$ requires the solution to a polynomial spectral factorization problem (see \cite{Daub}). The non-uniqueness of the solution leads to distinct wavelet bases, with some bases being more favorable than other bases. In particular, wavelet bases wherein the basis functions exhibit the least degree of asymmetry (i.e., those that are more nearly symmetric) are generally considered more favorable for signal processing applications such as noise reduction and deconvolution. Highly non-symmetric wavelet bases can lead to more noticeable signal distortion in these applications. Of relevance here is the fact that wavelet bases for which the basis functions are approximately symmetric are generated by a polynomial spectral factorization where the factor has approximately the same number of roots inside and outside the unit circle, i.e., a non-optimal spectral factorization. This type of factorization is sometimes known as a ``mid-phase" factorization (in contrast to a 'minimum-phase' factorization). The construction of Daubechies wavelet bases has been extended and generalized in several directions, some of which require the solution to a matrix spectral factorization problem rather than a polynomial one. Such generalization include  multiwavelet bases which consist of several distinct wavelet functions, and non-dyadic wavelet bases with rational dilation factors between 1 and 2 (see \cite{Sel}). As in the original Daubechies construction, the shape of the constructed basis functions depend substantially on the particular spectral factorization, as different spectral factorizations yield more or less favorable properties of the basis functions. In addition to the approximate symmetry of the basis functions, the time-localization or time-bandwidth product is also affected by the spectral factorization. In particular, the most favorable basis will generally be obtained via a 'non-optimal' matrix spectral factorization.

In the present paper we are going to explicitly characterize all polynomial non-optimal spectral factors of \eqref{11}.

In the scalar case, it follows from the standard proof of Fej\'{e}r-Riesz lemma that if $S$ has $2N$ distinct roots (non of them can appear on $\bT$ then, as every root on $\bT$ is of even multiplicity), then there exist totally $2^N$ different spectral factors, only one of them being optimal (Obviously we assume that two spectral factors are same if they differ only by a scalar multiplier with modulus 1). It is clear also how to deal with multiple roots, just combinatorial problem might become cumbersome.

In the matrix case, the factorization \eqref{1} implies that we can factorize the determinant
$$
\det  S(z)=\det S_+(z)\,\det\wdt{S_+}(z),
$$
and the question arise if the above reasoning about the Fej\'{e}r-Riesz lemma will be helpful to ``organize" all non-optimal polynomial spectral factors of \eqref{11}. On this direction we prove the following

\begin{theorem} Let $S$ be a Laurent polynomial matrix function \eqref{11}
  which is positive definite for almost everywhere on $\mathbb{T}$, and let $p$ be any  $($may be non-optimal$)$ scalar spectral factor of $\det S$, i.e.
 \begin{equation} \label{ssf}
p(z)\tld{p}(z)=\det S(z).
\end{equation}
   Then there exits a polynomial matrix function $P$ of the same degree $N$,
\begin{equation} \label{q2}
 P(z) = \sum_{n=0}^N B_nz^n,
\end{equation}
which is a $($non-optimal$)$ spectral factor of $S(z)$,
\begin{equation} \label{2.1}
P(z)\tld{P}(z)=S(z),
\end{equation}
such that
\begin{equation} \label{2.2}
\det P(z)=p(z),\;\;\;z\in \bC.
\end{equation}
\end{theorem}

Theorem 2 easily follows from the proof of Theorem 1 presented in \cite{E}, and we provide details in Section 2.

Now the question arises whether this factorization with the preassigned determinant is unique (up to a constant right unitary multiplier). A simple example in Section 3 shows that the uniqueness might fail if $\det S$ has multiple roots. Nevertheless,  the following uniqueness theorem holds in the case of simple roots of the determinant.

\begin{theorem}
Let $S$ and $p$ be the same as in Theorem 2, and suppose $\det S$ has only simple roots. Then there exists a unique $($up to a constant unitary right multiplier$)$ polynomial matrix \eqref{q2} such that \eqref{2.1} and \eqref{2.2} hold.
\end{theorem}

Theorem 3 together with the standard proof of the Fej\'{e}r-Riesz lemma implies

\begin{corollary}
Let $S$ be the same as in Theorem 3 and suppose $\det S$ is a Laurent polynomial of degree $L$.  Then $S$ has exactly $2^L$ spectral factors.
\end{corollary}

Of course, in the above statement, the spectral factors which differ from each other by a constant unitary multiplier are identified.

\smallskip

When this manuscript was completed, the authors learned from J. Ball [private communication] that the problem of non-optimal factorization with given determinant  was also considered by C. Hanselka and M. Schweighofer. J. Ball and L. Rodman, in their turn, proposed different approaches to this problem.

\section{Proof of Theorem 2}

Let $S_+(z)=\sum_{k=0}^N A_kz^k$ be an optimal spectral factor of $S$,
$$
S_+(z)\wdt{S_+}(z)=S(z),
$$
which existence is guaranteed by Theorem 1, and let
$$
\det S_+(z)=p_+(z).
$$
Since $\det S(z)=p(z)\tld{p}(z)=p_+(z)\wdt{p_+}(z)$, if $a$ is a zero of $p$, then either it is a zero of $p_+$ as well or $a^*=1/\ol{a}$ is a zero of $p_+$. In the former case we consider other zeros of $p$. In the latter case we transform polynomial matrix $S_+$ as follows:

Let $V$ be a constant unitary matrix such that the first column of $S_+(z)V$ vanishes at $z=a^*$. If we consider now $P_1(z)=S_+(z)V\cdot U(z)$, where
\begin{equation}\label{U}
U(z)=\diag[u(z),1,\ldots,1]\; \text{ with }\; u(z)=(z-a)/(1-\overline{a}z),
\end{equation}
 then $P_1$ remains polynomial (of the same degree $N$) matrix spectral factor of \eqref{11} and
$$
\det P_1(z)=p_+(z)\cdot u(z).
$$

In a similar manner we can change (if necessary) location of every zero of $p_+$ (moving them into symmetric with respect to $\bT$ points) and in the end we get the desired spectral factor $P$:
$$
P(z)=S_+(z)V_1\,U_1(z)\,V_2\,U_2(z)\ldots V_k\,U_k(z)
$$
satisfying \eqref{2.2}.

\section{Non uniqueness example for non optimal spectral factorizations}

Let $|a|<1$, and consider
$$
p(z)=(z-a)(z^{-1}-\ol{a})=(1-\ol{a}z)(1-az^{-1}).
$$
Let
$$
P_+(z)=\begin{pmatrix} z-a& 0\\0 &1-\ol{a}z \end{pmatrix}\;\;\text{ and }\;\;
R_+(z)=\begin{pmatrix} 1-\ol{a}z& 0\\0 & z-a\end{pmatrix}
$$
Then
$$
P_+\wdt{P_+}=R_+\wdt{R_+=}\begin{pmatrix} p(z)& 0\\0 & p(z)\end{pmatrix}
$$
and
$$
\det P_+(z)=\det R_+(z)=p(z).
$$
However
$$
R^{-1}_+(z)P_+(z)= \begin{pmatrix} u(z)& 0\\0 & u^{-1}(z)\end{pmatrix}
$$
with $u$  defined by \eqref{U}, is not a constant matrix.

\section{Proof of Theorem 3}
First we prove two lemmas from Linear Algebra.

\begin{lemma}
If $P$ is an $m\tm m$ polynomial matrix \eqref{q2}, and $a$ is a simple root of $\det P$, then
\begin{equation}\label{rank}
\rank P(a)=m-1.
\end{equation}
\end{lemma}
\begin{proof}
The proof follows from the Smith decomposition of $P$. Indeed
$$
P(z)=L(z)D(z)F(z),
$$
where $\det L(z)=\const$, $\det F(z)=\const$, and $D(z)=\diag[d_1(z),\dots,d_n(z)]$ with $d_i|d_{i+1}$ ($d_i$ divides $d_{i+1}$), $i=1,2,\ldots,m-1$. If $a$ is a simple root of $\det P$, it means that $d_i(a)\not=0$ for $i=1,2,\ldots,m-1$ and $d_n(a)=0$ (since $d_i|d_{i+1}$). Hence $\rank D(a)=m-1$ and \eqref{rank} follows.
\end{proof}

Let $\calU(m)$ be the set of unitary $m\tm m$ matrices with determinant 1.

For $M\in\bC^{m\tm m}$ let $\calN(M)$ and $\calR(M)$ be the nullspace and the column space of $M$, respectively.

\begin{lemma}
Let $M\in \bC^{m\tm m}$ and $\rank(M)=m-1$. Suppose $V_1, V_2\in \calU(m)$ and the first columns of $MV_1$ and $MV_2$ are both $\mbo\in\bC^{m\tm 1}$. Then
\begin{equation}\label{03}
V_1=V_2V,
\end{equation}
where $V$ has a block matrix form
\begin{equation}\label{04}
V=\diag(c,V_0) \;\text{ with } \;\;|c|=1\;\text{  and } \;V_0\in\calU(m-1).
\end{equation}
\end{lemma}
\begin{proof} Obviously \eqref{03} holds for some $V\in\calU(m)$. We only have to show that $V$ has the specific form \eqref{04}.

Let $v_1$ and $v_2$ be the first columns of $V_1$ and $V_2$, respectively. Since $v_1,v_2\in \calN(M)$ and $\dim \calN(M)=1$, we have
\begin{equation}\label{05}
v_1=cv_2
\end{equation}
for some $c\in\bC$. Since $\|v_1\|=\|v_2\|=1$, we have $|c|=1$.

Since the columns of $V_2$ are independent, and their linear combination with coefficients from the first column of $V$ yields $v_1$,  we can conclude, due to \eqref{05}, that $(c,0,0,\ldots,0)^T\in\bC^{m\tm 1}$ is the first column of $V$. Since other columns of $V$ are orthogonal to $(c,0,0,\ldots,0)^T$, we infer that the first row of $V$ is $(c,0,0,\ldots,0)\in\bC^{1\tm m}$. Thus $V$ has the desired block form, and obviously $V_0\in\calU(m-1)$ since $V\in\calU(m)$.
\end{proof}

\smallskip

We prove Theorem 3 by induction with respect to $\kappa=\card\{|z|<1: p(z)=0\}$.

If $\kappa=0$, then Theorem 3 is reduced to Theorem 1 and therefore true. Let us assume that the theorem is true for every  $p$ with $\kappa$ (simple) roots inside $\bT$, and let us prove it for $p$ having $\kappa+1$ roots inside $\bT$.

Let $P$ and $P_1$ be two (non-optimal) spectral factors of $S$ for which $\det P(z)=\det P_1(z)=p(z)$ holds, where the latter has $\kappa+1$ simple roots inside $\bT$.  Take $a$ inside $\bT$ such that $p(a)=0$. The functions $p$ and $\tld{p}$ have symmetric roots with respect to $\bT$, i.e.,
$p(a^*)=0 \Leftrightarrow  \tld{p}(a)=0$, where $a^*=1/\ol{a}$. Since $\det S$ has simple roots, it follows from \eqref{ssf} that $p(a^*)\not=0$.
 Consequently, the polynomial $q(z)=p(z)u^{-1}(z)$, where $u$ is defined by \eqref{U}, has $\kappa$ (simple) roots inside $\bT$, having also $z=a^*$ as a simple root. Let $S_0$ be a spectral factor of $S$ with determinant $q$ (its existence is proved in Theorem 2). Observe that
 \begin{equation}\label{rnk}
\rank S_0(a^*)=m-1
\end{equation}
 because of Lemma 1.

There exists $V\in\calU(m)$  such that the first column of $P(a)V$ is $\mbo\in\bC^{m\tm 1}$.
Then
$$
P(z)VU^{-1}(z),
$$
with $U$  defined by \eqref{U}, is a matrix polynomial spectral factor of $S$ with determinant $q(z)=p(z)u^{-1}(z)$ which has $\kappa$ roots inside $\bT$. By assumption of induction,
$$
P(z)\cdot V\cdot U^{-1}(z)\cdot W=S_0(z)
$$
for some $W\in\calU(m)$. Similarly,
$$
P_1(z)\cdot V_1\cdot U^{-1}(z)\cdot W_1=S_0(z).
$$
Thus we have
\begin{equation}\label{101}
P(z)V=S_0(z)W^{-1}U(z)
\end{equation}
and
\begin{equation}\label{102}
P_1(z)V_1=S_0(z)W_1^{-1}U(z).
\end{equation}

Since the left-hand side of the equation \eqref{101} is a polynomial matrix and $U(z)$ has a pole at $z=a^*$, it follows that the first column of $S_0(a^*)W^{-1}$ is $\mbo\in\bC^{m\tm 1}$. Similar arguments used for \eqref{102} imply that   $S_0(a^*)W_1^{-1}$ has the first column $\mbo\in\bC^{m\tm 1}$.
Hence we can use Lemma 2 to conclude that
\begin{equation}\label{16}
W^{-1}=W_1^{-1}\diag(c, W_0)
\end{equation}
with $W_0\in\calU(m-1)$. Using the diagonal structure of $\diag(c, W_0)$, we see that it commutes with $U(z)$. Consequently, taking into account \eqref{101}, \eqref{16} and \eqref{102}, we have
$P(z)V=S_0(z)W^{-1}U(z)=S_0(z)W_1^{-1}U(z)\diag(c, W_0)=P_1(z)V_1\diag(c, W_0)=P_1(z)V_0$, where $V_0\in \calU(m)$, which implies that $P$ and $P_1$ differ only by a constant unitary multiplier and Theorem 3 follows.

\def\cprime{$'$}
\providecommand{\bysame}{\leavevmode\hbox to3em{\hrulefill}\thinspace}
\providecommand{\MR}{\relax\ifhmode\unskip\space\fi MR }
\providecommand{\MRhref}[2]{%
  \href{http://www.ams.org/mathscinet-getitem?mr=#1}{#2}
}
\providecommand{\href}[2]{#2}


\begin{thebibliography}{10}

\bibitem{Sel}
{\.I}. Bayram and I.~W. Selesnick, \emph{Overcomplete discrete wavelet
  transforms with rational dilation factors}, IEEE Trans. Signal Process.
  \textbf{57} (2009), no.~1, 131--145. \MR{2674791}

\bibitem{CG91}
P.~E. Caines and L{\'a}szl{\'o} Gerencs{\'e}r, \emph{A simple proof for a
  spectral factorization theorem}, IMA J. Math. Control Inform. \textbf{8}
  (1991), no.~1, 39--44. \MR{1109605}

\bibitem{Callier}
F.~M. Callier, \emph{On polynomial matrix spectral factorization by
  symmetric extraction}, IEEE Trans. Automat. Control \textbf{30} (1985),
  no.~5, 453--464. \MR{789691}

\bibitem{Daub}
I. Daubechies, \emph{Ten lectures on wavelets}, CBMS-NSF Regional
  Conference Series in Applied Mathematics, vol.~61, Society for Industrial and
  Applied Mathematics (SIAM), Philadelphia, PA, 1992. \MR{1162107 (93e:42045)}

\bibitem{DGK}
P.~Delsarte, Y.~Genin, and Y.~Kamp, \emph{A simple approach to spectral
  factorization}, IEEE Trans. Circuits and Systems \textbf{25} (1978), no.~11,
  943--946. \MR{508983}

\bibitem{E}
L.~Ephremidze, \emph{An elementary proof of the polynomial matrix spectral
  factorization theorem}, Proc. Roy. Soc. Edinburgh Sect. A \textbf{144}
  (2014), no.~4, 747--751. \MR{3233753}

\bibitem{EJL09}
L.~Ephremidze, G.~Janashia, and E.~Lagvilava, \emph{A simple proof of the
  matrix-valued {F}ej\'er-{R}iesz theorem}, J. Fourier Anal. Appl. \textbf{15}
  (2009), no.~1, 124--127. \MR{2491029 (2010c:47050)}

\bibitem{EL10}
L.~Ephremidze and E.~Lagvilava, \emph{Remark on outer analytic
  matrix-functions}, Proc. A. Razmadze Math. Inst. \textbf{152} (2010), 29--32.
  \MR{2663529}

\bibitem{Kai99}
T.~Kailath, B.~Hassibi, and A.~H. Sayed, \emph{Linear estimation},
  Prentice-Hall, Inc., Englewood Cliffs, N.J., 1999, Prentice-Hall Information
  and System Sciences Series.

\bibitem{Ros}
M. Rosenblatt, \emph{A multi-dimensional prediction problem}, Ark. Mat.
  \textbf{3} (1958), 407--424. \MR{0092332}

\bibitem{Ru}
W. Rudin, \emph{Real and complex analysis}, third ed., McGraw-Hill Book
  Co., New York, 1987. \MR{924157}

\bibitem{Wie57}
N.~Wiener and P.~Masani, \emph{The prediction theory of multivariate stochastic
  processes. {I}. {T}he regularity condition}, Acta Math. \textbf{98} (1957),
  111--150. \MR{0097856 (20 \#4323)}

\bibitem{Youla}
D.~C. Youla, \emph{On the factorization of rational matrices}, IRE Trans.
  \textbf{IT-7} (1961), 172--189. \MR{0132753}

\end{thebibliography}
\end{document}